\documentclass[11pt,a4paper]{amsart}
\usepackage{amsmath}
\usepackage{geometry}
\usepackage{hyperref}
\usepackage[utf8]{inputenc}
\usepackage[english]{babel}
\usepackage[T1]{fontenc}
\usepackage{url}
\usepackage{mathrsfs}
\usepackage{ifthen}
\usepackage{xparse}
\usepackage{diagbox}

\usepackage{tikz}
\usetikzlibrary{knots}
\usetikzlibrary{decorations.pathreplacing}
\usepackage{xcolor}
\usepackage{graphicx}

\usepackage{amssymb}
\usepackage{lscape}
\usepackage{multirow}
\usepackage{wrapfig} 

\usepackage{caption}
\usepackage{enumitem}

\newtheorem{thm}{Theorem}

\newtheorem{corollary}{Corollary}

\theoremstyle{definition}
\newtheorem{definition}{Definition}

\newtheorem*{convention*}{Conventions}

\theoremstyle{remark}
\newtheorem{remark}{Remark}

\newcommand{\oeislink}[1]{\href{https://oeis.org/#1}{\textcolor{blue}{\underline{#1}}}}

\newcommand{\R}{\mathbb{R}}

\newcommand \Mset  {\mathfrak{M}}     %Set of equivalence classes of meanders
\newcommand \M  {\mathcal{M}}       %Meander numbers
\newcommand \Mclset  {\overline{\mathfrak{M}}}     %Set of equivalence classes of meanders
\newcommand \Mcl  {\overline{\mathcal{M}}}       %Meander numbers
\newcommand \Mirr  {\mathcal{M}^{(Ir)}}  %Number of irreducible meanders
\newcommand \Miirr  {\mathcal{M}^{(IIr)}}
\newcommand \Mis  {\mathcal{M}^{(IS)}} %Number of iterated snakes
\title{Singular meanders}
\author{Y. Belousov}
\email{bus99@yandex.ru}
\address{Saint Petersburg State University}
\thanks{This work was funded by the Russian Science Foundation, grant No. 25-11-00251, \url{https://rscf.ru/project/25-11-00251/}.}

\begin{document}
\begin{abstract}
    The problem of enumerating meanders --- pairs of simple plane curves with transverse intersections --- was formulated about forty years ago and is still far from solved. Recently, it was discovered that meanders admit a factorization into prime components. This factorization naturally leads to a broader class of objects, which we call singular meanders, in which tangential intersections between the curves are also allowed. In the present paper we initiate a systematic study of singular meanders: we develop a basic combinatorial framework, point out connections with other combinatorial objects and known integer sequences, and completely enumerate several natural families of singular meanders.
\end{abstract}

\maketitle
	
\section*{Introduction}
A meander is a configuration of a pair of simple curves in a disk. Examples of meanders are shown in Fig.~\ref{fig:examples-of-meanders}.
The problem of counting meanders was initially formulated by V.~Arnol'd in~\cite{A88}; a similar problem of counting closed meanders (under the name of planar permutation) was formulated by P.~Rosenstiehl in~\cite{R84}. Meanders appear in different areas of mathematics, and we refer readers to the survey~\cite{Z23} for both connections and historical context.
At present, the meander problem is far from being solved: counts are known only for small numbers of intersections; no subexponential-time algorithms for computing these numbers are known; and the asymptotic growth rate remains conjectural.

\begin{figure}[ht]
    \centering
    \resizebox{0.9\linewidth}{!}{
    \begin{tikzpicture}[scale = 5]
\draw[thick] (0, 0) to (1, 0);
\draw[thick] (0.0333263, 0.179487)
	to[out = 0, in = 90, distance = 4.83322] (0.384615, 0)
	to[out = -90, in = -90, distance = 2.89993] (0.153846, 0)
	to[out = 90, in = 90, distance = 0.966644] (0.230769, 0)
	to[out = -90, in = -90, distance = 0.966644] (0.307692, 0)
	to[out = 90, in = 90, distance = 2.89993] (0.0769231, 0)
	to[out = -90, in = -90, distance = 4.83322] (0.461538, 0)
	to[out = 90, in = 90, distance = 4.83322] (0.846154, 0)
	to[out = -90, in = -90, distance = 2.89993] (0.615385, 0)
	to[out = 90, in = 90, distance = 0.966644] (0.692308, 0)
	to[out = -90, in = -90, distance = 0.966644] (0.769231, 0)
	to[out = 90, in = 90, distance = 2.89993] (0.538462, 0)
	to[out = -90, in = -90, distance = 4.83322] (0.923077, 0)
	to[out = 90, in = 180, distance = 0.966644] (0.966674, 0.179487);
\draw[help lines] (0.5, 0) circle (0.5);
\draw[fill] (0.0333263, 0.179487) circle (0.0166667);
\draw[fill] (0.966674, 0.179487) circle (0.0166667);
\draw[fill] (0, 0) circle (0.0166667);
\draw[fill] (1, 0) circle (0.0166667);
\begin{scope}[shift = {(1.7, 0)}]
\draw[thick] (0, 0) to (1, 0);
\draw[thick] (0.0740823, 0.261905)
	to[out = 0, in = 90, distance = 2.69279] (0.214286, 0)
	to[out = -90, in = -90, distance = 0.897598] (0.285714, 0)
	to[out = 90, in = 90, distance = 8.07838] (0.928571, 0)
	to[out = -90, in = -90, distance = 0.897598] (0.857143, 0)
	to[out = 90, in = 90, distance = 6.28318] (0.357143, 0)
	to[out = -90, in = -90, distance = 2.69279] (0.571429, 0)
	to[out = 90, in = 90, distance = 0.897598] (0.642857, 0)
	to[out = -90, in = -90, distance = 6.28318] (0.142857, 0)
	to[out = 90, in = 90, distance = 0.897598] (0.0714286, 0)
	to[out = -90, in = -90, distance = 8.07838] (0.714286, 0)
	to[out = 90, in = 90, distance = 2.69279] (0.5, 0)
	to[out = -90, in = -90, distance = 0.897598] (0.428571, 0)
	to[out = 90, in = 90, distance = 4.48799] (0.785714, 0)
	to[out = -90, in = 180, distance = 2.69279] (0.925918, -0.261905);
\draw[help lines] (0.5, 0) circle (0.5);
\draw[fill] (0.0740823, 0.261905) circle (0.0166667);
\draw[fill] (0.925918, -0.261905) circle (0.0166667);
\draw[fill] (0, 0) circle (0.0166667);
\draw[fill] (1, 0) circle (0.0166667);
\end{scope}
\end{tikzpicture}
}
    \caption{Examples of (non-singular) meanders.}
    \label{fig:examples-of-meanders}
\end{figure}
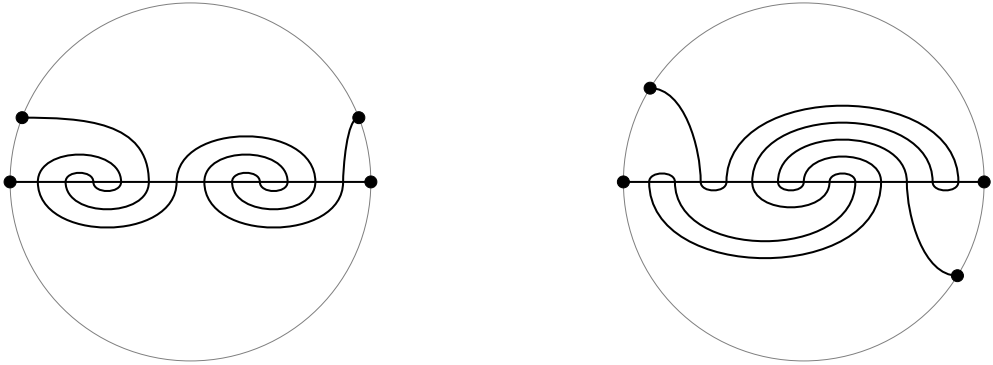

In a recent paper~\cite{B24}, we introduced the factorization of meanders into prime components, and in this context more general objects naturally arise --- singular meanders. These are meanders in which not only transverse intersections but also tangency points are allowed (see the example in Fig.~\ref{fig:examples-of-singular-meanders}). 
Although singular meanders have not been studied before, preliminary enumerative data suggest links with other combinatorial families (some of which were already mentioned in~\cite{B24}). In the present note, we provide a preliminary analysis of singular meanders: we discuss some of their properties, and provide some numerical observations.

The paper is organized as follows. In Section~\ref{sec:definitions}, we formally define singular meanders and closed singular meanders and introduce the necessary related notions. In Section~\ref{sec:proof-and-corollaries}, we discuss elementary combinatorial properties of singular meanders. In Section~\ref{sec:known-families}, we list all families of singular meanders that we are currently able to describe completely.

\subsection*{Acknowledgments}
The author thanks Andrei Malyutin for many fruitful discussions.

\section{Definitions}\label{sec:definitions}
In this section, we present all notation related to singular meanders that we will need later (we refer the reader to~\cite[Section~2]{B24} for a more detailed discussion). In Subsection~\ref{subsec:closed-meanders}, we define closed singular meanders.

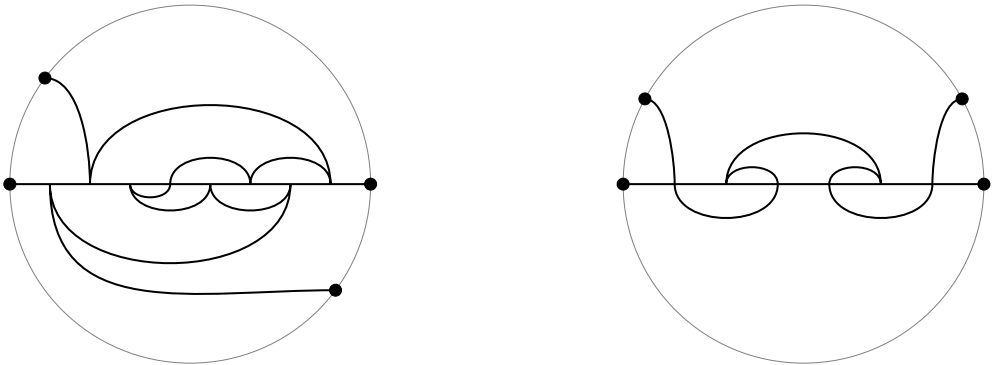
\begin{figure}[ht]
    \centering    
    \resizebox{0.9\linewidth}{!}{
    \begin{tikzpicture}[scale = 5]
    \draw[thick] (0, 0) to (1, 0);
    \draw[thick] (0.0972488, 0.296296)
    	to[out = 0, in = 90, distance = 2.79253] (0.222222, 0)
    	to[out = 90, in = 90, distance = 8.37758] (0.888889, 0)
    	to[out = 90, in = 90, distance = 2.79253] (0.666667, 0)
    	to[out = 90, in = 90, distance = 2.79253] (0.444444, 0)
    	to[out = -90, in = -90, distance = 1.39626] (0.333333, 0)
    	to[out = -90, in = -90, distance = 2.79253] (0.555556, 0)
    	to[out = -90, in = -90, distance = 2.79253] (0.777778, 0)
    	to[out = -90, in = -90, distance = 8.37758] (0.111111, 0)
    	to[out = -90, in = 180, distance = 11.1701] (0.902751, -0.296296);
    \draw[help lines] (0.5, 0) circle (0.5);
    \draw[fill] (0.0972488, 0.296296) circle (0.0166667);
    \draw[fill] (0.902751, -0.296296) circle (0.0166667);
    \draw[fill] (0, 0) circle (0.0166667);
    \draw[fill] (1, 0) circle (0.0166667);
    \begin{scope}[shift = {(1.7, 0)}]
    \draw[thick] (0, 0) to (1, 0);
    \draw[thick] (0.0603289, 0.238095)
    	to[out = 0, in = 90, distance = 1.7952] (0.142857, 0)
    	to[out = -90, in = -90, distance = 3.59039] (0.428571, 0)
    	to[out = 90, in = 90, distance = 1.7952] (0.285714, 0)
    	to[out = 90, in = 90, distance = 5.38559] (0.714286, 0)
    	to[out = 90, in = 90, distance = 1.7952] (0.571429, 0)
    	to[out = -90, in = -90, distance = 3.59039] (0.857143, 0)
    	to[out = 90, in = 180, distance = 1.7952] (0.939671, 0.238095);
    \draw[help lines] (0.5, 0) circle (0.5);
    \draw[fill] (0.0603289, 0.238095) circle (0.0166667);
    \draw[fill] (0.939671, 0.238095) circle (0.0166667);
    \draw[fill] (0, 0) circle (0.0166667);
    \draw[fill] (1, 0) circle (0.0166667);
    \end{scope}
    \end{tikzpicture}
    }
    \caption{Examples of singular meanders.}
    \label{fig:examples-of-singular-meanders}
\end{figure}

\begin{definition}\label{def:meander}
A \emph{singular meander} $M$ is a triple $(D, (p_1, p_2, p_3, p_4), (l, m))$ of
\begin{itemize}
    \item a 2-dimensional closed disk $D$;
    \item four distinct points $p_1, p_2, p_3, p_4$ on the boundary $\partial D$ such that $p_3$ and $p_4$ lie in the same connected component of $\partial D \setminus \{p_1, p_2\}$;
    \item the images $m$ and $l$ of smooth proper embeddings of the segment $[0;\,1]$ into $D$ such that $\partial m = \{p_1, p_3\}$, $\partial l = \{p_2, p_4\}$, and $m \cap l$ is a non-empty finite set.
\end{itemize}
The intersection points of $l$ and $m$ are called \emph{intersections of $M$}. 
\end{definition}

\begin{definition}
We say that two singular meanders
$$
M = (D, (p_1, p_2, p_3, p_4), (l, m))
\quad\text{and}\quad
M' = (D', (p_1', p_2', p_3', p_4'), (l', m'))
$$
are \emph{equivalent} if there exists a homeomorphism $f \colon D \to D'$ (not necessarily orientation-preserving) such that $f(m) = m'$, $f(l) = l'$, and $f(p_i) = p_i'$ for each $i = 1, \dots, 4$. 
\end{definition}

\begin{definition}\label{def:order}
Let $M$ be a singular meander and write $[M]$ for its equivalence class.
For any $M' \in [M]$, let $n_{\mathrm{t}}(M')$ and $n_{\mathrm{nt}}(M')$ be,
respectively, the numbers of transverse and non-transverse intersections
of the arcs $l$ and $m$ of $M'$.

The \emph{order} of $M$ is the pair
$$
(n, k)
=
\left(
\max_{M' \in [M]} n_{\mathrm{t}}(M'),
\;
\min_{M' \in [M]} n_{\mathrm{nt}}(M')
\right).
$$
If the order of $M$ is $(n, k)$, the \emph{total order} of $M$ is $n + k$.
We denote by $\Mset_{n,k}$ the set of equivalence classes of singular meanders of
order $(n, k)$ and by $\M_{n,k}$ its cardinality.
\end{definition}

\begin{table}[h]
    \begin{tabular}{|c|c|c|c|c|c|c|c|c|c|c|}
        \hline
        \diagbox[width=3em]{$\textbf{n}$}{$\textbf{k}$}   & \textbf{0} & \textbf{1} & \textbf{2}  & \textbf{3}   & \textbf{4}   & \textbf{5}    & \textbf{6}     & \textbf{7}     & \textbf{8}      & \textbf{9}
        \\ \hline
        \textbf{0} & 0 & 1 & 1  & 1   & 1   & 1    & 1     & 1     & 1      & 1 
        \\ \hline
        \textbf{1} & 1 & 4 & 14 & 48  & 166 & 584  & 2092  & 7616  & 28102  & 104824
        \\ \hline
        \textbf{2} & 1 & 7 & 36 & 166 & 730 & 3138 & 13328 & 56204 & 235854 & 986010 
        \\ \hline
        \textbf{3} & 2 & 24 & 188 & 1224 & 7202 & 39808 & 210992 & 1085248 & 5457284 & 26959616
        \\ \hline
    \end{tabular}
    \caption{Values of $\M_{n,k}$ for small $n$ and $k$.}
    \label{tab:mnk}
\end{table}

\begin{definition}\label{def:submeander}
Let
$$
M = (D, (p_1, p_2,p_3,p_4), (l, m)) \quad \text{and} \quad M' = (D', (p_1', p_2',p_3',p_4'),(l', m'))
$$
be two singular meanders. We say that $M'$ is a \emph{submeander} of $M$ if
\begin{itemize}
    \item $D' \subseteq D$;
    \item $m' = D' \cap m$;
    \item $l' = D' \cap l$;
    \item $p_1' = \gamma_m(t_1)$, where $\gamma_m \colon [0;\,1] \to D$ is any injective continuous map such that $\gamma_m([0;\,1]) = m$, $\gamma_m(0) = p_1$, and
    $$
        t_1 = \min\{t \in [0;\,1] \mid \gamma_m(t) \in D'\};
    $$
    \item Let $S = \gamma_m^{-1}(l) \cap [0;\, t_1]$.
    If $S \neq \varnothing$, let $t_q = \max S$ and $q = \gamma_m(t_q)$; otherwise, set $q := p_2$.
    Choose an injective continuous map $\gamma_l \colon [0;\,1] \to D$ with $\gamma_l([0;\,1]) = l$ such that $\gamma_l^{-1}(q) < t$ for all $t \in \gamma_l^{-1}(D')$.
    Define $p_2'$ by $p_2' := \gamma_l(t_2)$, where
    $$
        t_2 = \min\{t \in [0;\,1] \mid \gamma_l(t) \in D'\}.
    $$
\end{itemize}
\end{definition}

\begin{definition}
Let $M$ be a singular meander, and let $M'$ and $M''$ be two of its submeanders. We say that $M'$ and $M''$ are \emph{equivalent with respect to $M$} if they contain exactly the same subset of intersection points of $M$.
\end{definition}

\begin{definition} \label{def:insertion}
Let
$$M = (D, (p_1, p_2,p_3,p_4), (l, m))\quad \text{and} \quad M' = (D', (p_1', p_2',p_3',p_4'),(l', m'))$$
be two singular meanders of order $(n, k)$ and $(n',k')$ respectively, and let 
$$M'' = (D'', (p_1'', p_2'',p_3'',p_4''), (l'', m''))$$
be a submeander of $M$  of order $(n'',k'')$ and total order $1$ such that $n'' \equiv n' \pmod{2}$. 
Consider a map $f\colon\partial D'' \to \partial D'$ such that $f({p}_i'')=p'_i$ for each $i=1,\dots,4$. 
Then there is a well-defined meander $$\tilde{M} = (\tilde{D}, (p_1, p_2, p_3, p_4), (\tilde{l}, \tilde{m}))$$ where  
\begin{itemize}
    \item $\tilde{D} = \big(D\setminus \operatorname{Int}(D'')\big) \cup_f D'$;
    \item $\tilde{l} = \big(l\setminus \operatorname{Int}(D''\cap l)\big)\cup_f l'$;
    \item $\tilde{m} = \big(m\setminus \operatorname{Int}(D''\cap m)\big)\cup_f m'$.
\end{itemize}
We say that $\tilde{M}$ is obtained by the \emph{insertion of $M'$ into $M$ at $M''$}.
\end{definition}

\begin{definition}
Let $M$ be a singular meander of order $(n, k)$. 
$M$ is said to be \emph{irreducible} if its total order is greater than two and there are precisely $n + k + 1$ pairwise non-equivalent submeanders of $M$. We denote by $\Mirr_{n,k}$ the number of equivalence classes of irreducible singular meanders of order $(n, k)$.  \\
$M$ is called a \emph{snake} if its total order is greater than one and there are precisely $\frac{(n + k)(n + k + 1)}{2}$ pairwise non-equivalent submeanders of $M$.  \\
$M$ is called an \emph{iterated snake} if it can be obtained from a snake by finitely many insertions of snakes. We denote by $\Mis_{n,k}$ the number of equivalence classes of iterated snakes of order $(n, k)$.  
\end{definition}

\begin{table}[h]
    \centering
        \begin{tabular}{|c|c|c|c|c|c|c|c|c|c|c|c|}
        \hline
        \diagbox[width=3em]{$\textbf{n}$}{$\textbf{k}$} & \textbf{3} & \textbf{4} & \textbf{5} & \textbf{6} & \textbf{7} & \textbf{8} & \textbf{9} & \textbf{10} & \textbf{11} & \textbf{12} & \textbf{13}
        \\ \hline
        \textbf{1} & 0 & 2  & 0  & 8   & 8   & 36   & 72    & 212   & 528    & 1438   & 3816    \\ \hline
        \textbf{2} & 2 & 0  & 12 & 14  & 72  & 162  & 530   & 1452  & 4314   & 12402  & 36246   \\ \hline
        \textbf{3} & 0 & 0  & 16 & 48  & 240 & 884  & 3328  & 11960 & 42112  & 145860 & 497856  \\ \hline
        \textbf{4}     & 0 & 10 & 36 & 210 & 884 & 3744 & 14950 & 57904 & 218790 & 809016 & 2942240 \\ \hline
        \end{tabular}
    \caption{Values of $\Mirr_{n,k}$ for small $n$ and $k$.}
    \label{tab:mirrnk}
\end{table}

\begin{thm}[\cite{B24}]\label{thm:decomp}
    Each singular meander can be canonically constructed using iterated snakes and irreducible singular meanders.
\end{thm}

\begin{remark}
    The numbers of singular meanders, iterated snakes and irreducible singular meanders of small total order can be found in~\cite{Bcode}.
\end{remark}

\subsection*{Conventions}
\begin{enumerate}
    \item Without loss of generality, assume that for every singular meander $M$ of order $(n, k)$, we have $n_{\mathrm{t}}(M) = n$ and $n_{\mathrm{nt}}(M) = k$.
    \item We draw singular meanders in such a way that:
    \begin{itemize}
        \item $D$ is a Euclidean disk;
        \item $l$ is a horizontal diameter with $p_2$ at the left end;
        \item $p_1$ is placed above $p_2$.
    \end{itemize}
    This allows us to omit the labels $l$, $m$, and $p_1, \dots, p_4$.
\end{enumerate}

\subsection{Closed singular meanders}\label{subsec:closed-meanders}
\begin{definition}\label{def:closed-meander}
A \emph{closed singular meander} $M$ is a tuple $(D, (p_1, p_2), l, m)$ of
\begin{itemize}
    \item a $2$-dimensional closed disk $D$;
    \item two distinct points $p_1, p_2$ on the boundary $\partial D$;
    \item the image $l$ of a smooth proper embedding of the segment $[0;\,1]$ into $D$ such that $\partial l = \{p_1, p_2\}$;
    \item the image $m$ of a smooth proper embedding of the circle $S^1$ into $D$ such that $m$ and $l$ intersect (not necessarily transversely) in a non-empty finite set of points.
\end{itemize}
The intersection points of $m$ and $l$ are called \emph{intersections of $M$}. 
\end{definition}

\begin{definition}
We say that two closed singular meanders
$$
M = (D, (p_1, p_2), l, m)
\quad\text{and}\quad
M' = (D', (p_1', p_2'), l', m')
$$
are \emph{equivalent} if there exists an orientation-preserving homeomorphism $f \colon D \to D'$ such that $f(m) = m'$, $f(l) = l'$, and $f(p_i) = p_i'$ for each $i = 1, 2$. 
\end{definition}

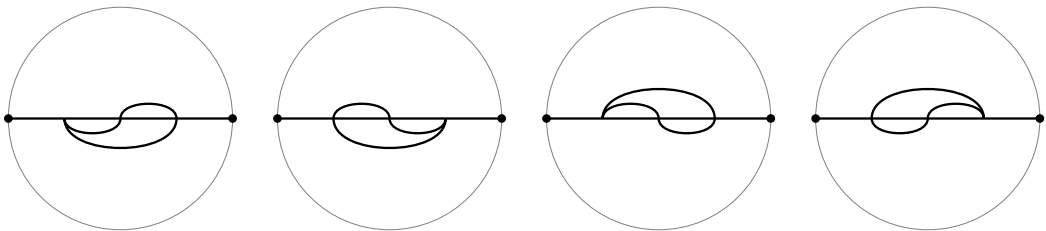
\begin{figure}[h]
    \centering
    \resizebox{0.95\linewidth}{!}{
    \begin{tikzpicture}[scale = 2.7]
\draw[thick] (0, 0) to (1, 0);
\draw[thick] (0.4*5/4, 0)
	to[out = -90, in = -90, distance = 2.51327] (0.2*5/4, 0)
	to[out = -90, in = -90, distance = 5.02655] (0.6*5/4, 0)
	to[out = 90, in = 90, distance = 2.51327] (0.4*5/4, 0);
\draw[help lines] (0.5, 0) circle (0.5);
\draw[fill] (0, 0) circle (0.0166667);
\draw[fill] (1, 0) circle (0.0166667);
\begin{scope}[shift = {(1.2, 0)}]
\draw[thick] (0, 0) to (1, 0);
\draw[thick] (0.2*5/4, 0)
	to[out = -90, in = -90, distance = 5.02655] (0.6*5/4, 0)
	to[out = -90, in = -90, distance = 2.51327] (0.4*5/4, 0)
	to[out = 90, in = 90, distance = 2.51327] (0.2*5/4, 0);
\draw[help lines] (0.5, 0) circle (0.5);
\draw[fill] (0, 0) circle (0.0166667);
\draw[fill] (1, 0) circle (0.0166667);
\end{scope}
\begin{scope}[shift = {(2.4, 0)}]
\draw[thick] (0, 0) to (1, 0);
\draw[thick] (0.2*5/4, 0)
	to[out = 90, in = 90, distance = 2.51327] (0.4*5/4, 0)
	to[out = -90, in = -90, distance = 2.51327] (0.6*5/4, 0)
	to[out = 90, in = 90, distance = 5.02655] (0.2*5/4, 0);
\draw[help lines] (0.5, 0) circle (0.5);
\draw[fill] (0, 0) circle (0.0166667);
\draw[fill] (1, 0) circle (0.0166667);
\end{scope}
\begin{scope}[shift = {(3.6, 0)}]
\draw[thick] (0, 0) to (1, 0);
\draw[thick] (0.2*5/4, 0)
	to[out = -90, in = -90, distance = 2.51327] (0.4*5/4, 0)
	to[out = 90, in = 90, distance = 2.51327] (0.6*5/4, 0)
	to[out = 90, in = 90, distance = 5.02655] (0.2*5/4, 0);
\draw[help lines] (0.5, 0) circle (0.5);
\draw[fill] (0, 0) circle (0.0166667);
\draw[fill] (1, 0) circle (0.0166667);
\end{scope}
\end{tikzpicture}
}
    \caption{Examples of non-equivalent closed singular meanders.}
    \label{fig:closed-non-equivalent}
\end{figure}

\begin{remark}
    Note that a homeomorphism in the definition of equivalence for singular meanders is not required to be orientation-preserving, in contrast to the definition of equivalence for closed singular meanders. This difference arises because we want to consider the four closed singular meanders in Fig.~\ref{fig:closed-non-equivalent} as pairwise non-equivalent.
\end{remark}

\begin{definition}
    For closed singular meanders, the \emph{order} and the \emph{total order} are defined verbatim as in Definition~\ref{def:order}. We denote by $\Mclset_{n,k}$ the set of equivalence classes of closed singular meanders of order $(n, k)$ and by $\Mcl_{n,k}$ its cardinality.
\end{definition}

\begin{remark}
    As in the case of classical meanders, one can introduce singular versions of other related objects: semimeanders (configurations of a ray and a curve, see~\cite{DFGG95}) and stamp foldings (configurations of a segment and a curve, see~\cite{koehler1968folding}). One could also consider more general objects allowing both singular intersection points and multiple connected components. 
\end{remark}

\section{Combinatorics of singular meanders}\label{sec:proof-and-corollaries}
Throughout this section, whenever we mention a singular meander, we implicitly choose the geometric representative satisfying our drawing conventions from Section~\ref{sec:definitions}.
\begin{thm}\label{thm:closed-and-open}
    Let $n$ be a positive even integer, and $k$ be a non-negative integer. Then
    $$\Mcl_{n,k} = \M_{n-1, k} + 2\M_{n, k-1}.$$
\end{thm}
\begin{proof}
 \begin{figure}[ht]
        \centering
        \resizebox{0.7\linewidth}{!}{
        \begin{tikzpicture}[scale = 3]
        \draw[thick] (0, 0) to (1, 0);
        \draw[thick] (0.0520968, 0.222222)
        	to[out = 0, in = 90, distance = 4.18879] (0.333333, 0)
        	to[out = 90, in = 90, distance = 2.09439] (0.166667, 0)
        	to[out = -90, in = -90, distance = 4.18879] (0.5, 0)
        	to[out = 90, in = 90, distance = 4.18879] (0.833333, 0)
        	to[out = 90, in = 90, distance = 2.09439] (0.666667, 0)
        	to[out = -90, in = 180, distance = 4.18879] (0.947903, -0.222222);
        \draw[help lines] (0.5, 0) circle (0.5);
        \draw[fill] (0.0520968, 0.222222) circle (0.0166667);
        \draw[fill] (0.947903, -0.222222) circle (0.0166667);
        \draw[fill] (0, 0) circle (0.0166667);
        \draw[fill] (1, 0) circle (0.0166667);
        \coordinate (Aend) at (1.25,0);
        
        \begin{scope}[shift={(2,0)}]
        \draw[thick] (0, 0) to (1, 0);
        \draw[thick] (0.25*8/7, 0)
        	to[out = 90, in = 90, distance = 1.5708] (0.125*8/7, 0)
        	to[out = -90, in = -90, distance = 3.14159] (0.375*8/7, 0)
        	to[out = 90, in = 90, distance = 3.14159] (0.625*8/7, 0)
        	to[out = 90, in = 90, distance = 1.5708] (0.5*8/7, 0)
        	to[out = -90, in = -90, distance = 3.14159] (0.75*8/7, 0)
        	to[out = 90, in = 90, distance = 6.28318] (0.25*8/7, 0);
        \draw[help lines] (0.5, 0) circle (0.5);
        \draw[fill] (0, 0) circle (0.0166667);
        \draw[fill] (1, 0) circle (0.0166667);
        \coordinate (Bstart) at (-0.25,0);
        \end{scope}
        \draw[->, thick] (Aend) -- (Bstart);
        \node (a) at (-0.5, -0) {(a)};
        \end{tikzpicture}
        }
        
        \vspace{0.2cm}
        \resizebox{0.7\linewidth}{!}{
        \begin{tikzpicture}[scale = 3]
        \draw[thick] (0, 0) to (1, 0);
        \draw[thick] (0.0454703, 0.208333)
        	to[out = 0, in = 90, distance = 4.71239] (0.375, 0)
        	to[out = -90, in = -90, distance = 1.5708] (0.25, 0)
        	to[out = 90, in = 90, distance = 1.5708] (0.125, 0)
        	to[out = -90, in = -90, distance = 6.28318] (0.625, 0)
        	to[out = 90, in = 90, distance = 1.5708] (0.75, 0)
        	to[out = -90, in = -90, distance = 1.5708] (0.875, 0)
        	to[out = 90, in = 90, distance = 4.71239] (0.5, 0)
        	to[out = 90, in = 180, distance = 6.28318] (0.95453, 0.208333);
        \draw[help lines] (0.5, 0) circle (0.5);
        \draw[fill] (0.0454703, 0.208333) circle (0.0166667);
        \draw[fill] (0.95453, 0.208333) circle (0.0166667);
        \draw[fill] (0, 0) circle (0.0166667);
        \draw[fill] (1, 0) circle (0.0166667);
        \coordinate (Aend1) at (1.25,0.2);
        \coordinate (Aend2) at (1.25,-0.2);
        
        \begin{scope}[shift={(2, 0.6)}]
        \draw[thick] (0, 0) to (1, 0);
        \draw[thick] (0.3*10/9, 0)
        	to[out = -90, in = -90, distance = 1.25664] (0.2*10/9, 0)
        	to[out = 90, in = 90, distance = 1.25664] (0.1*10/9, 0)
        	to[out = -90, in = -90, distance = 5.02655] (0.5*10/9, 0)
        	to[out = 90, in = 90, distance = 1.25664] (0.6*10/9, 0)
        	to[out = -90, in = -90, distance = 1.25664] (0.7*10/9, 0)
        	to[out = 90, in = 90, distance = 3.76991] (0.4*10/9, 0)
        	to[out = 90, in = 90, distance = 5.02655] (0.8*10/9, 0)
        	to[out = 90, in = 90, distance = 6.28318] (0.3*10/9, 0);
        \draw[help lines] (0.5, 0) circle (0.5);
        \draw[fill] (0, 0) circle (0.0166667);
        \draw[fill] (1, 0) circle (0.0166667);
        \coordinate (Bstart1) at (-0.25,-0.20);
        \end{scope}
        
        \begin{scope}[shift={(2,-0.6)}, yscale=-1]
        \draw[thick] (0, 0) to (1, 0);
        \draw[thick] (0.3*10/9, 0)
        	to[out = -90, in = -90, distance = 1.25664] (0.2*10/9, 0)
        	to[out = 90, in = 90, distance = 1.25664] (0.1*10/9, 0)
        	to[out = -90, in = -90, distance = 5.02655] (0.5*10/9, 0)
        	to[out = 90, in = 90, distance = 1.25664] (0.6*10/9, 0)
        	to[out = -90, in = -90, distance = 1.25664] (0.7*10/9, 0)
        	to[out = 90, in = 90, distance = 3.76991] (0.4*10/9, 0)
        	to[out = 90, in = 90, distance = 5.02655] (0.8*10/9, 0)
        	to[out = 90, in = 90, distance = 6.28318] (0.3*10/9, 0);
        \draw[help lines] (0.5, 0) circle (0.5);
        \draw[fill] (0, 0) circle (0.0166667);
        \draw[fill] (1, 0) circle (0.0166667);
        \coordinate (Bstart2) at (-0.25, -0.20);
        \end{scope}
        \draw[->, thick] (Aend1) -- (Bstart1);
        \draw[->, thick] (Aend2) -- (Bstart2);
        \node (b) at (-0.5, -0) {(b)};
        \end{tikzpicture}
        }
        \caption{Transforming singular meanders into closed singular meanders.}
        \label{fig:open-meanders-to-closed}
    \end{figure}
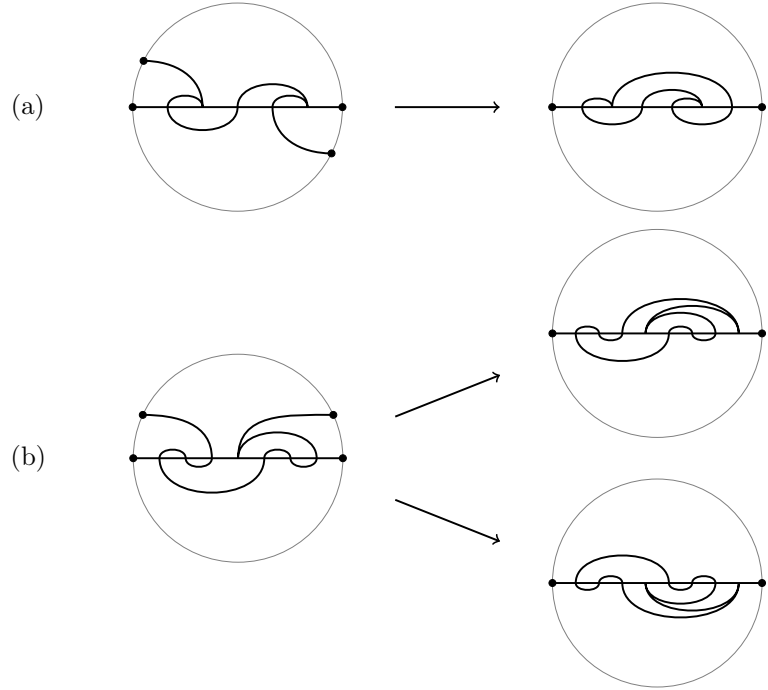
    Let $n$ and $k$ be as in the statement of the theorem. Note that if $k=0$ then $\M_{n,-1} = 0$ and we recover a well-known formula, connecting the numbers of open and closed meanders. 
    
    For each singular meander of order $(n-1, k)$, we can construct a closed singular meander of order $(n, k)$ by connecting the ends of $m$ and adding one more transverse intersection point on the right; see Fig.~\ref{fig:open-meanders-to-closed}~(a). For each singular meander of order $(n, k-1)$, there are two different ways to construct closed singular meanders of order $(n, k)$: either connect the ends of $m$ and add one more non-transverse intersection point on the right, or first reflect along $l$ and then add one more non-transverse intersection point on the right; see Fig.~\ref{fig:open-meanders-to-closed}~(b). (These two closed singular meanders are clearly non-equivalent.) It is clear that non-equivalent singular meanders lead to non-equivalent closed singular meanders. Finally, each closed singular meander can be obtained from a singular meander in this way.
\end{proof}

\begin{thm} \label{thm:main}
    Let $k$ be a positive integer and $n$ be a positive even integer. Then
    $$k \M_{n-1,k} = 2n \M_{n, k-1}.$$
\end{thm}

\begin{proof}    
    Let $n$ and $k$ be as in the theorem.     
    The finite group $\mathbb{Z}/(n+k)\mathbb{Z}$ acts on $\Mclset_{n,k}$ by cyclically moving the leftmost intersection point to the right; see Fig.~\ref{fig:permutation-group-action}. For a closed singular meander $M$, let $O$ be its orbit under this action. Then $O$ contains $\frac{n+k}{d_O}$ elements, where $d_O \geqslant 1$ denotes the cardinality of the stabilizer of any element in $O$ (for instance, in Fig.~\ref{fig:permutation-group-action} we have $d_O = 2$). Moreover, among these elements, exactly $\frac{n}{d_O}$ are obtained from singular meanders of order $(n-1, k)$ and $\frac{k}{d_O}$ are obtained from singular meanders of order $(n, k-1)$. Thus
    \begin{align*}
        \M_{n-1, k} &= \sum_{O} \frac{n}{d_O},\\
        2\M_{n, k-1} &= \sum_{O} \frac{k}{d_O},\\
        \frac{\M_{n-1, k}}{n} &= \frac{2\M_{n, k-1}}{k},
    \end{align*}
    where the sums are taken over all orbits $O$. Hence $k \M_{n-1,k} = 2n \M_{n, k-1}$, as claimed. \qedhere

    \begin{figure}[ht]
        \newcommand \localshift {1.3}
        \newcommand \dist{0.4} 
        \newcommand \anglearrow {40}
        \newcommand \arrowshift {0.5}
        \centering
        \resizebox{0.6\linewidth}{!}{
        \begin{tikzpicture}[scale = 3]
        \draw[thick] (0, 0) to (1, 0);
        \draw[thick] (0.25*8/7, 0)
        	to[out = 90, in = 90, distance = 4.71239] (0.625*8/7, 0)
        	to[out = 90, in = 90, distance = 3.14159] (0.375*8/7, 0)
        	to[out = -90, in = -90, distance = 3.14159] (0.125*8/7, 0)
        	to[out = -90, in = -90, distance = 4.71239] (0.5*8/7, 0)
        	to[out = -90, in = -90, distance = 3.14159] (0.75*8/7, 0)
        	to[out = 90, in = 90, distance = 6.28318] (0.25*8/7, 0);
        \draw[help lines] (0.5, 0) circle (0.5);
        \draw[fill] (0, 0) circle (0.0166667);
        \draw[fill] (1, 0) circle (0.0166667);
        
        \draw [line width = 2pt, ->] 
            (0.5 + 0.5*\dist  + 0.86603*\arrowshift, 0 - 0.86603*\dist + 0.5*\arrowshift) 
            to [out = -60 + \anglearrow, in = 120 - \anglearrow, looseness = 1.3] 
            (0.5 + 0.57735*\localshift - 0.50000*\dist + 0.86603*\arrowshift, -\localshift + 0.86603*\dist+ 0.5*\arrowshift);
        \draw [line width = 2pt, ->] 
            (0.5 + 0.57735*\localshift - 1.00000*\dist, -\localshift - \arrowshift) 
            to [out = -180 + \anglearrow, in = 0 - \anglearrow, looseness = 1.3]
            (0.5 - 0.57735*\localshift + 1.00000*\dist, -\localshift - \arrowshift);
        \draw [line width = 2pt, ->] 
            (0.5 - 0.57735*\localshift + 0.50000*\dist - 0.86603*\arrowshift, -\localshift + 0.86603*\dist+ 0.5*\arrowshift) 
            to [out = 60 + \anglearrow, in = -120 - \anglearrow, looseness = 1.3]
            (0.5 - 0.50000*\dist - 0.86603*\arrowshift, 0 - 0.86603*\dist+ 0.5*\arrowshift);

        \begin{scope}[shift={(0.57735*\localshift, -\localshift)}]
        \draw[thick] (0, 0) to (1, 0);
        \draw[thick] (0.125*8/7, 0)
        	to[out = 90, in = 90, distance = 4.71239] (0.5*8/7, 0)
        	to[out = 90, in = 90, distance = 3.14159] (0.25*8/7, 0)
        	to[out = -90, in = -90, distance = 6.28318] (0.75*8/7, 0)
        	to[out = -90, in = -90, distance = 4.71239] (0.375*8/7, 0)
        	to[out = -90, in = -90, distance = 3.14159] (0.625*8/7, 0)
        	to[out = 90, in = 90, distance = 6.28318] (0.125*8/7, 0);
        \draw[help lines] (0.5, 0) circle (0.5);
        \draw[fill] (0, 0) circle (0.0166667);
        \draw[fill] (1, 0) circle (0.0166667);
        \end{scope}
        
        \begin{scope}[shift={(-0.57735*\localshift, -\localshift)}]
        \draw[thick] (0, 0) to (1, 0);
        \draw[thick] (0.125*8/7, 0)
        	to[out = -90, in = -90, distance = 6.28318] (0.625*8/7, 0)
        	to[out = -90, in = -90, distance = 4.71239] (0.25*8/7, 0)
        	to[out = -90, in = -90, distance = 3.14159] (0.5*8/7, 0)
        	to[out = 90, in = 90, distance = 3.14159] (0.75*8/7, 0)
        	to[out = 90, in = 90, distance = 4.71239] (0.375*8/7, 0)
        	to[out = 90, in = 90, distance = 3.14159] (0.125*8/7, 0);
        \draw[help lines] (0.5, 0) circle (0.5);
        \draw[fill] (0, 0) circle (0.0166667);
        \draw[fill] (1, 0) circle (0.0166667);
        \end{scope}
        \end{tikzpicture}
        }
        \caption{Action of the group $\mathbb{Z}/(n+k)\mathbb{Z}$ on closed singular meanders.}
        \label{fig:permutation-group-action}
    \end{figure}
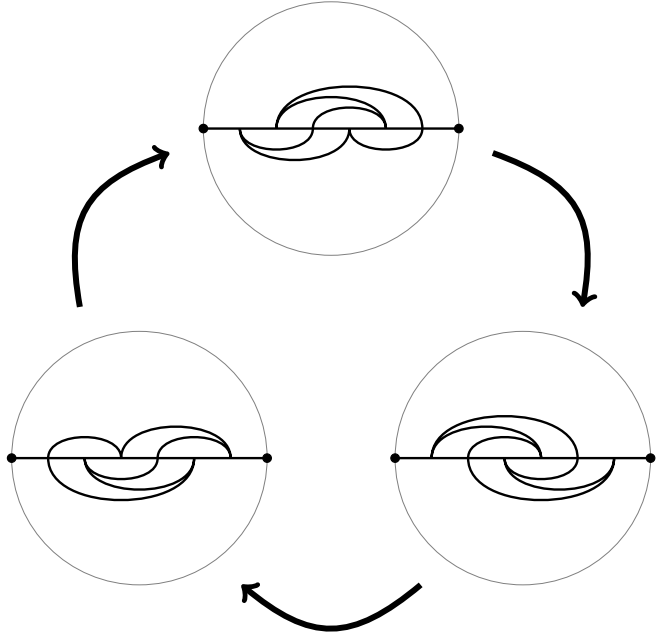
\end{proof}

\begin{remark}\label{rem:gf-derivative}
    The statement of Theorem~\ref{thm:main} can be reformulated as follows. Let
    $$
        \phi(x,t) := \sum_{n,k \geqslant 0} \M_{n,k} x^n t^k
    $$
    be the generating function for $\{\M_{n,k}\}_{n,k \geqslant 0}$, and let
    $$
        \phi_{(odd)}(x,t) := \sum_{n,k \geqslant 0} \M_{2n+1,k} x^{2n+1} t^k
        \quad\text{and}\quad
        \phi_{(even)}(x,t) := \sum_{n,k \geqslant 0} \M_{2n,k} x^{2n} t^k
    $$
    be its odd and even parts with respect to the variable $x$. Then 
    $$
        2 \,\partial_x \phi_{(even)}(x,t) = \partial_t \phi_{(odd)}(x,t).
    $$
\end{remark}

%\begin{question}
%    Does the generating function for the numbers of singular meanders $\phi(x,t)$ satisfy other non-trivial differential equations?
%\end{question}

\begin{figure}[ht]
    \centering
    \includegraphics[width=0.7\linewidth]{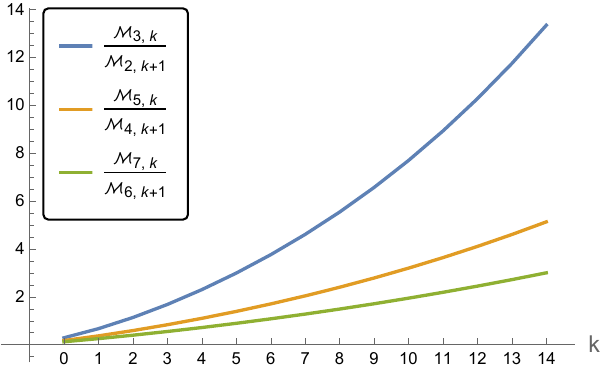}
    \caption{Numerical behavior of the quotient $\M_{n,k}/\M_{n-1,k+1}$ for a fixed odd $n$.}
    \label{fig:meander-quotient}
\end{figure}

\begin{remark}
    Although for any fixed odd $n$ the quotient $\frac{\M_{n,k-1}}{\M_{n-1,k}}$ is well approximated by a third-degree polynomial in $k$ (see Fig.~\ref{fig:meander-quotient}), we were not able to find an exact relation.
\end{remark}

\begin{corollary}\label{cor:main-irreducible}
    Let $k$ be a positive integer and $n$ be a positive even integer. Then
    $$k \Mirr_{n-1,k} = 2n \Mirr_{n, k-1}.$$
\end{corollary}
\begin{proof}
    It is easy to see that the group action described in the proof of Theorem~\ref{thm:main} preserves irreducibility in the following sense: if a closed singular meander $M$ is an image of an irreducible singular meander, then all other closed singular meanders in the orbit of $M$ are also images of irreducible singular meanders (see the example in Fig.~\ref{fig:permutation-group-action}).
\end{proof}

\begin{corollary}
    Let $n$ be a positive even integer and $k$ a positive integer coprime to $n$. Then
    \begin{align*}
        &\M_{n-1, k} \equiv 0 \pmod{2n},\\
        &\M_{n, k-1} \equiv 0 \pmod{k},\\
        &\Mirr_{n-1, k} \equiv 0 \pmod{2n},\\
        &\Mirr_{n, k-1} \equiv 0 \pmod{k}.
    \end{align*}
\end{corollary}

\section{Some families of singular meanders with a complete enumeration} \label{sec:known-families}
\subsection{Singular meanders with a small number of transverse intersections}
\begin{thm}\label{thm:m1k}
    $$
    \sum_{k \geqslant 0} \M_{1, k} t^k = \frac{1}{(1 - 2t)\sqrt{1 - 4t}}.
    $$
\end{thm}
\begin{proof}
    Let $M = (D, (p_1, p_2, p_3, p_4), (l, m))$ be a singular meander of order $(1, k)$. Without loss of generality, we may assume that (see Fig.~\ref{fig:m1k}~(a)) 
    \begin{enumerate}
        \item $D$ is the Euclidean disk $\{(x,y) \in \R^2 \mid x^2 + y^2 \leqslant 1\}$;
        \item $p_1 = (0,1)$;
        \item $p_2 = (-1,0)$;
        \item $p_3 = (0,-1)$;
        \item $p_4 = (1,0)$.        
    \end{enumerate}
    After performing an isotopy, we may assume that $l \cup m$ is contained in a small neighborhood of 
    $$
        C_M := I_X \cup I_Y \cup \bigcup_{i=1}^k I_i,
    $$
    where $I_X = \{(t,0) \mid t \in [-1;\,1]\}$, $I_Y = \{(0,t) \mid t \in [-1;\,1]\}$, and $\{I_i\}$ is a set of pairwise disjoint segments (which we call \emph{chords}) with one endpoint on $I_X$ and the other endpoint on $I_Y$; see Fig.~\ref{fig:m1k}~(b). We call such $C_M$ a \emph{carcass} (with $k$ chords); see Fig.~\ref{fig:m1k}~(c).

    \begin{figure}[ht]
        \centering
        \resizebox{0.9\linewidth}{!}{
        \begin{tikzpicture}[scale = 3]
            \draw[thick] (0, 0) to (1, 0);
            \draw[thick] (0.5, 0.5)
            	to[out = -90, in = 90, looseness = 2] (0.125, 0)
            	to[out = 90, in = 90, distance = 3.14159] (0.375, 0)
            	to[out = 90, in = 90, distance = 1.5708] (0.5, 0)
            	to[out = 90, in = 90, distance = 1.5708] (0.625, 0)
            	to[out = -90, in = -90, distance = 4.71239] (0.25, 0)
            	to[out = -90, in = -90, distance = 6.28318] (0.75, 0)
            	to[out = -90, in = -90, distance = 1.5708] (0.875, 0)
            	to[out = -90, in = 90, looseness = 2] (0.5, -0.5);
            \draw[help lines] (0.5, 0) circle (0.5);
            \draw[fill] (0.5, 0.5) circle (0.0166667);
            \draw[fill] (0.5, -0.5) circle (0.0166667);
            \draw[fill] (0, 0) circle (0.0166667);
            \draw[fill] (1, 0) circle (0.0166667);
            \node (a) at (0.5, -0.75) {(a)};

        \begin{scope}[shift = {(1.25, 0)}]
            \draw[thick] (0, 0) to (1, 0);
            \draw [thick] (0.5, 0.5) 
                to (0.5, 0.5-0.0714)
                to (0.1, -0.0)
                to (0.5, 0.5 - 1.5*0.0714)
                to (0.5, 0.5 - 3*0.0714)
                to (0.3, -0.0)
                to (0.5, 0.5 - 3.5*0.0714)
                to (0.5, 0.5 - 5*0.0714)
                to (0.4, 0.0)
                to (0.5, 0.5 - 5.5*0.0714)
                to (0.5, 0)
                to (0.5, -1.5*0.0714)
                to (0.2, -0.0)
                to (0.5, -2*0.0714)
                to (0.5, -3*0.0714)
                to (0.5 + 0.1667, 0.0)
                to (0.5, -3.5*0.0714)
                to (0.5, -5*0.0714)
                to (0.5 + 2*0.1667, 0.0)
                to (0.5, -5.5*0.0714)
                to (0.5, -0.5);                

            \draw[help lines] (0.5, 0) circle (0.5);
            \draw[fill] (0.5, 0.5) circle (0.0166667);
            \draw[fill] (0.5, -0.5) circle (0.0166667);
            \draw[fill] (0, 0) circle (0.0166667);
            \draw[fill] (1, 0) circle (0.0166667);
            \node (a) at (0.5, -0.75) {(b)};
        \end{scope}

        \begin{scope}[shift = {(2.5, 0)}]
        \draw[help lines] (0.5, 0) circle (0.5);
        
        \draw[thick] (0, 0) to (1, 0);
        \draw[thick] (0.5, 0.5) to (0.5, -0.5);

        \draw[thick] (0.1, 0) to (0.5, 0.375);
        \draw[thick] (0.2, 0) to (0.5, -0.125);
        \draw[thick] (0.3, 0) to (0.5, 0.25);
        \draw[thick] (0.4, 0) to (0.5, 0.125);
        \draw[thick] (0.6667, 0) to (0.5, -0.25);
        \draw[thick] (0.8334, 0) to (0.5, -0.375);
        \node (a) at (0.5, -0.75) {(c)};        
        \end{scope}
        
        \end{tikzpicture}
        }
        \caption{Example of a singular meander of order $(1,6)$ and its corresponding carcass.}
        \label{fig:m1k}
    \end{figure}
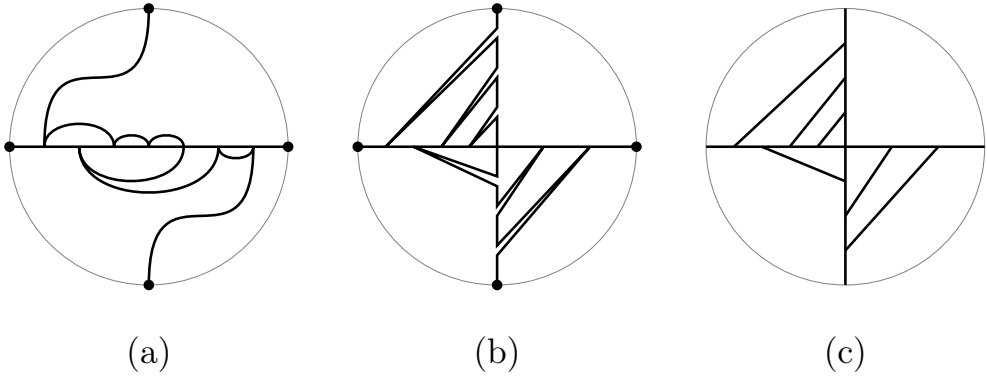
    
    The carcass is uniquely defined up to isotopy of $D$ that keeps $I_X \cup I_Y$ setwise fixed. Non-equivalent singular meanders lead to non-isotopic carcasses. Thus the number of non-equivalent singular meanders of order $(1,k)$ equals the number of non-isotopic carcasses with $k$ chords. The number of carcasses with $k$ chords is given by
    $$
    \sum_{\substack{x_1, x_2, x_3, x_4 \geqslant 0\\x_1 + x_2 + x_3 + x_4 = k}} 
    \binom{x_1 + x_2}{x_1}\binom{x_2 + x_3}{x_2}\binom{x_3 + x_4}{x_3}\binom{x_4 + x_1}{x_4},
    $$
    where $x_i$ is the number of chords lying in the $i$-th quadrant (for example, in Fig.~\ref{fig:m1k}~(c) $x_1 = 0$, $x_2 = 3$, $x_3 = 1$, and $x_4 = 2$), and $\binom{x_1 + x_2}{x_1}$ corresponds to all possible arrangements of the endpoints of chords on the segment
    $\{(0,t) \mid t \in (0,1)\}$, with the other binomial coefficients interpreted analogously. Thus
    \begin{align*}
    \sum_{k \geqslant 0} \M_{1, k} t^k
        &= \sum_{k \geqslant 0}\sum_{\substack{x_1, x_2, x_3, x_4 \geqslant 0\\x_1 + x_2 + x_3 + x_4 = k}} 
           \binom{x_1 + x_2}{x_1}\binom{x_2 + x_3}{x_2}\binom{x_3 + x_4}{x_3}\binom{x_4 + x_1}{x_4} \, t^k \\
        &= \frac{1}{(1 - 2t)\sqrt{1 - 4t}}.
    \end{align*}    
\end{proof}

\begin{remark}
    The sequence $\{\M_{1,k}\}_{k \geqslant 0}$ coincides with the OEIS sequence \oeislink{A082590}~\cite{oeis}, which counts the number of $k$-letter words over the alphabet $\{1, 2, 3, 4\}$ having as many occurrences of the substring (consecutive subword) $[1, 2]$ as of $[1, 3]$. An explicit bijection between equivalence classes of singular meanders with a single transverse intersection and such words will be given in a subsequent paper (together with several other interpretations of singular meanders with a single transverse intersection). 
\end{remark}

\begin{remark}
    A technique similar to that in the proof of Theorem~\ref{thm:m1k} can be applied to obtain a formula for $\M_{n,k}$ for other values of $n$. However, the formulas quickly become too cumbersome and cannot be easily expressed via generating functions. For example,
    $$
        \M_{3, k}
        = 2\sum_{\substack{x_0,\dots,x_3 \geqslant 0 \\ y_0,\dots,y_3 \geqslant 0 \\ u_0,\dots,u_3 \geqslant 0 \\ \sum_{i=0}^3(x_i + y_i + u_i)=k}}
        \left(
        \prod_{i=0}^3
        \binom{x_i + x_{i+1} + u_i}{u_i}
        \prod_{i=0}^3
        \binom{y_i + y_{i+1} + u_i}{u_i}
        \right),
    $$
    where the indices of $x_i$ and $y_i$ are considered modulo $4$.
\end{remark}

\begin{thm}
    $$
    \sum_{k\geqslant 0} \M_{2, k}t^k = \frac{1-3 t}{(1-2 t)^2\sqrt{(1-4 t)^3}}.
    $$
\end{thm}
\begin{proof}
    From Theorem~\ref{thm:main} with $n=2$ it follows that
    $$
        k\,\M_{1,k} = 4\,\M_{2,k-1},
    $$
    and hence
    $$
        \M_{2,k} = \frac{k+1}{4}\,\M_{1,k+1}.
    $$
    Using Theorem~\ref{thm:m1k}, we obtain
    $$
        \sum_{k\geqslant 0} \M_{2, k}t^k
        = \frac{1}{4}\,\partial_t\!\left( \frac{1}{(1 - 2t)\sqrt{1 - 4t}}\right)
        =  \frac{1-3 t}{(1-2 t)^2\sqrt{(1-4 t)^3}}.
    $$
\end{proof}

\subsection{Iterated snakes}
%Iterated snakes form a rather simple class of singular meanders, allowing a complete enumeration. Let 
%$$
%F(x,t) := \sum_{n,k \geqslant 0} \Mis_{n,k} x^n t^k
%$$
%be the generating function for iterated snakes. Then $F(x,t)$ can be found from the following equation (see~\cite{B24}):
%\begin{equation}\label{eq:gf-for-iter-snakes}
%v^3 + x v^2 + 2v - 4(t-1)^2(v+x) = 0 \quad \text{where} \quad v = 2F(x,t) + x + 2.    
%\end{equation}
Iterated snakes form a rather simple class of singular meanders, allowing a complete enumeration. Let
$$
F(x,t) := \sum_{n,k \geqslant 0} \Mis_{n,k}\,x^n t^k
$$
be the generating function for iterated snakes. In this subsection singular meanders of total order $1$ are regarded as snakes, in accordance with~\cite{B24}.
Set
$$
U(x,t) := 1 + F(x,t).
$$
Then $U(0,0)=1$, and $U(x,t)$ is the unique formal power series satisfying the equation
\begin{equation}\label{eq:gf-for-iter-snakes}
	(1-t)U^3 - \bigl(4(1-t)^2 - 2(1-t)x + 1\bigr)U^2 + (6(1-t)-x)U - 2 = 0.
\end{equation}
%Equivalently, $F(x,t)$ is determined by~\eqref{eq:gf-for-iter-snakes} via $U=1+F$.
Some subsequences of the array $\left\lbrace\Mis_{n,k}\right\rbrace_{n,k\ge 0}$ coincide with other combinatorial sequences that can be found in the OEIS~\cite{oeis}%\footnote{For an exact match, singular meanders of total order 1 should also be considered as snakes.}
:
\begin{itemize}  
    \item $\left\lbrace\Mis_{0, k}\right\rbrace_{k>0}$ is \oeislink{A000012} (this is just the sequence of all ones);
    \item $\left\lbrace\Mis_{1,k}\right\rbrace_{k\geqslant 0}$ is \oeislink{A007070};
    \item $\left\lbrace\Mis_{2,k}\right\rbrace_{k\geqslant 0}$ is \oeislink{A181292};
    \item $\left\lbrace\Mis_{n,0}\right\rbrace_{n \geqslant 0}$ is \oeislink{A007165}.
\end{itemize}
The first three identifications follow since the corresponding subsequences have the same initial terms and satisfy the same recurrence relation (which follows from Eq.~\eqref{eq:gf-for-iter-snakes}). The sequence~\oeislink{A007165} counts the number of $P$-graphs, which are just graph-theoretic reformulations of iterated snakes; see~\cite{R86}.

\subsection{Irreducible meanders}
\begin{thm}
    Let $\varphi(x)$ be Euler's totient function. Then, for all integers $n,k \geqslant 0$,
    \begin{enumerate}      
    \item $\Mirr_{n,k} = 0$ for $k < 3$ or $n < 1$;
    \item $\Mirr_{2n+1,3} = 0$;
    \item $\Mirr_{2n+2,3} = \varphi(n+5) - 2$;
    \item $\Mirr_{2n+1,4} = (n+1)\bigl(\varphi(n+5) - 2\bigr)$;
    \item $\displaystyle \sum_{k \geqslant 0}\Mirr_{1, k} t^k = \frac{1}{8}\left(\sqrt{\frac{1-3t}{1+t}} - 1\right)^4$;
    \item $\displaystyle \sum_{k \geqslant 0}\Mirr_{2, k} t^k = \frac{1 - 2t - \sqrt{\frac{1-3t}{1+t}}}{\sqrt{(1-3t)(1+t)^5}}$.
    \end{enumerate}
\end{thm}
\begin{proof}
    Cases~\textup{(1)}--\textup{(3)} were proved in~\cite{B24}. Case~\textup{(4)} follows from case~\textup{(3)} and Theorem~\ref{thm:main}. Analogously, case~\textup{(6)} follows from case~\textup{(5)} and Theorem~\ref{thm:main}. The only thing that remains to prove is case~\textup{(5)}.

    To prove case~\textup{(5)}, we introduce two additional classes of singular meanders. A singular meander is said to be \emph{arborescent}\footnote{We use this term because the poset of its submeanders (see~\cite{B24}) is a tree.} if it can be obtained from an irreducible singular meander by finitely many insertions of irreducible singular meanders. The generating function for arborescent singular meanders of order $(1,k)$ can be easily expressed in terms of the generating function for irreducible singular meanders of order $(1,k)$. Let $\Miirr_{n,k}$ be the number of non-equivalent arborescent singular meanders of order $(n,k)$, and set
\begin{align*}
    F(t) &:= \sum_{k \geqslant 0} \Mirr_{1,k} t^k,\\
    G(t) &:= \sum_{k \geqslant 0} \Miirr_{1,k} t^k = \sum_{r \geqslant 1} F(t)^r = \frac{F(t)}{1 - F(t)}.
\end{align*}

A singular meander is called a \emph{B-meander} if it can be obtained by inserting singular meanders of order $(0,p)$ into arborescent singular meanders. Since for each $p > 0$ there is a unique equivalence class of singular meanders of order $(0,p)$, each arborescent singular meander of order $(n,r)$ leads to $\binom{k-1}{r-1}$ non-equivalent B-meanders of order $(n,k)$. Thus the generating function $B(t)$ for the numbers of B-meanders of order $(1,k)$ is given by
$$
    B(t)
    = \sum_{k \geqslant 0} \sum_{r=1}^k \Miirr_{1,r} \binom{k-1}{r-1} t^k
    = \sum_{r \geqslant 1} \Miirr_{1,r} \left(\frac{t}{1-t}\right)^r
    = G\!\left(\frac{t}{1-t}\right).
$$

From Theorem~\ref{thm:decomp} it follows that each singular meander of order $(1,k)$ is uniquely obtained via a sequential insertion of B-meanders and iterated snakes(excluding snake of order $(1,0)$), where iterated snakes are not inserted into iterated snakes and B-meanders are not inserted into B-meanders. Since $A(t)$ includes the trivial iterated snake of order $(1,0)$, the generating function for inserted iterated snakes is $A(t)-1$. Therefore
\begin{equation}\label{eq:gf}
	C(t) = A(t) + \frac{A(t)^2 B(t)}{1-\left(A(t)-1\right)B(t)},
\end{equation}
where 
\begin{align*}
	C(t) &:= \sum_{k \geqslant 0} \M_{1,k} t^k = \frac{1}{(1 - 2t)\sqrt{1 - 4t}},\\
	A(t) &:= \sum_{k \geqslant 0} \Mis_{1,k} t^k = \frac{1}{2 t^2 - 4 t + 1}.
\end{align*}
%From Theorem~\ref{thm:decomp} it follows that each singular meander of order $(1,k)$ is uniquely obtained via a sequential insertion of B-meanders and iterated snakes, where iterated snakes are not inserted into iterated snakes and B-meanders are not inserted into B-meanders. At the level of generating functions this leads to the following equation:
%\begin{equation}\label{eq:gf}
%    C(t) = \frac{A(t) + B(t) + 2A(t)B(t)}{1 - A(t)B(t)},
%\end{equation}
%where 
%\begin{align*}
%    C(t) &:= \sum_{k \geqslant 0} \M_{1,k} t^k = \frac{1}{(1 - 2t)\sqrt{1 - 4t}},\\
%    A(t) &:= \sum_{k \geqslant 0} \Mis_{1,k} t^k = \frac{1}{2 t^2 - 4 t + 1}.
%\end{align*}
From Eq.~\eqref{eq:gf} one finds $B(t)$, then $G(t)$, and finally $F(t)$, which completes the proof of case~\textup{(5)} and hence of the theorem.
\end{proof}

\bibliography{biblio}
\bibliographystyle{alpha}

\end{document}